\documentclass{birkjour}
\usepackage[english]{babel}
\usepackage{subfigure}
\usepackage[usenames, dvipsnames]{xcolor}
\usepackage[T1]{fontenc}
\usepackage{tikz}
\usepackage{pgfplots}

  \usepackage{cite}
\usepackage{amsfonts,amsmath,amsxtra,amsthm,amssymb,latexsym,mathrsfs}

 \newcommand\dom{\operatorname{dom}}
\DeclareMathOperator{\rad}{rad}

\DeclareMathOperator{\re}{Re}
 
 \DeclareMathOperator{\Span}{span}
    
  \DeclareMathOperator{\im}{Im}
  \DeclareMathOperator{\ran}{Ran}
  \DeclareMathOperator{\logg}{Log}
   \renewcommand{\log}{\textrm{Log}}
\newtheorem{theorem}{Theorem}[section]
 
 \newtheorem{lemma}[theorem]{Lemma}
 \newtheorem{proposition}[theorem]{Proposition}
 \theoremstyle{definition}
 \newtheorem{definition}[theorem]{Definition}
 \theoremstyle{remark}
 \newtheorem{remark}[theorem]{Remark}
 
 \numberwithin{equation}{section}
\renewcommand\proofname{Proof}

\begin{document}

\title[NLS-log equation with $\delta$-interaction]{Stability of standing waves for NLS-log  equation with $\delta$-interaction}

\author{Jaime Angulo Pava}

\address{%
 Rua do Mat\~{a}o 1010\\
CEP 05508-090\\
 S\~{a}o Paulo, SP\\
Brazil}

\email{angulo@ime.usp.br}

\thanks{J. Angulo was supported partially
by Grant CNPq/Brazil, N. Goloshchapova was supported by Fapesp under the projects 2012/50503-6 and 2016/02060-9}
\author{Nataliia Goloshchapova}
\address{%
 Rua do Mat\~{a}o 1010\\
CEP 05508-090\\
 S\~{a}o Paulo, SP\\
Brazil}
\email{nataliia@ime.usp.br}


\date{February 11, 2017}

\begin{abstract}
 We study analytically  the orbital stability of the standing waves with a peak-Gausson profile for a nonlinear logarithmic Schr\"o\-dinger  equation with $\delta$-interaction (attractive and repulsive). A major difficulty is to compute the number of negative eigenvalues of the linearized operator around the standing wave. This is
overcome by the perturbation method, the continuation arguments, and the theory of extensions of symmetric  operators.
 \end{abstract}
 
 \maketitle
 
 \section{Introduction}
In 1976 Bialynicki-Birula and Mycielski   \cite{BM} built a model of nonlinear wave mechanics based on the following Schr\"odinger equation with a logarithmic non-linearity  (NLS-log equation  henceforth)
 \begin{equation}\label{NLSL}
i\partial_t u+\Delta u + u\log|u|^2=0,
\end{equation}
 where $u=u(t, x):\, \mathbb{R}\times\mathbb{R}^n\rightarrow \mathbb{C}$, $n\geq 1$. This equation has been proposed in order to obtain a nonlinear equation which helped to quantify departures from the strictly linear regime, preserving in any number of dimensions some fundamental aspects of quantum mechanics, such as separability and additivity of total energy of noninteracting subsystems. The NLS-log equation admits applications to dissipative systems \cite{HR}, quantum mechanics, quantum optics \cite{BSS}, nuclear physics \cite{H}, transport and diffusion phenomena (for example, magma transport) \cite{DFGL},
open quantum systems, effective quantum gravity, theory of superfluidity, and Bose-Einstein condensation (see \cite{H, Zlo10}  and the references therein). 
We refer to \cite{Caz83, Caz80} for a study of existence and uniqueness of the solutions
to the associated Cauchy problem in a suitable functional framework, as well
as for a study of the asymptotic behavior of its solutions and their orbital stability.

In  this paper we study  the following nonlinear logarithmic Schr\"odinger equation with $\delta$-interaction (NLS-log-$\delta$ henceforth) on the line
\begin{equation}\label{NLS_delta}
i\partial_t u-\mathcal H_\gamma^\delta u +u\log|u|^2=0.
\end{equation}
 Here $u=u(t, x):\,\mathbb{R}\times \mathbb{R}\rightarrow \mathbb{C}$,  $\gamma\in\mathbb{R}\setminus\{0\}$,  and  $\mathcal H_\gamma^\delta$ is the self-adjoint operator on $L^2(\mathbb{R})$ defined by
\begin{equation}\label{H_delta}
\begin{split}
&\mathcal H_\gamma^\delta =-\frac{d^2}{dx^2},\\
&\dom(\mathcal H_\gamma^\delta)=\{f\in H^1(\mathbb R)\cap H^2(\mathbb R\setminus\{0\}): f'(0+)-f'(0-)=-\gamma f(0)\}.
\end{split}
\end{equation}

The operator  $\mathcal H_\gamma^\delta$  corresponds to the formal  expression  $l_\gamma^\delta=-\frac{d^2}{dx^2}-\gamma \delta$ (see \cite{AlbGes05}  for details). Equation \eqref{NLS_delta} can be viewed as a model of a  singular interaction between nonlinear wave and an inhomogeneity. The delta potential can be used to model an impurity, or defect, localized at the origin.  Formally the  NLS-log-$\delta$ model can be described by  the following  problem 
 \begin{equation*}
 \left \{ 
\begin{split}
&i\partial_{t}u(t, x)+ \partial^2_{x}u(t, x)= -u(t, x)\log |u(t, x)|^2,\quad x\neq0, \;t\in\mathbb R,\\
&\lim_{x\to 0^+}[u(t, x)-u(t, -x)]=0, \\
&\lim_{x\to 0^+}[\partial_x u(t, x)-\partial_x u(t, -x)]=-\gamma u(t, 0),\\
&\lim_{x\to \pm \infty} u(t, x)=0.
\end{split}\right.
\end{equation*}
A similar formal model  with power nonlinearity has been introduced in \cite{CMR}. 

Our aim  is  to investigate an orbital stability of standing wave solutions  $u(t, x)=e^{i\omega t}\varphi_{\omega,\gamma}$ for  equation \eqref{NLS_delta} with  {\it peak-Gausson} profile
\begin{equation}\label{gausson}
\varphi_{\omega,\gamma}(x)=e^{\frac{\omega+1}{2}} e^{-\frac{1}{2} (|x|+\frac{\gamma}{2})^2}.
 \end{equation} 
The main stability  result of this paper is the following.  
\begin{theorem}\label{main}  Let  $\gamma\neq 0$ and $\varphi_{\omega,\gamma}$ be defined by   \eqref{gausson}. Let also $\widetilde W$ be defined by \eqref{tilde_W}.
Then the  following  assertions hold.
\begin{itemize} 
\item[$(i)$] If $\gamma>0$, then the standing wave $e^{i\omega t}\varphi_{\omega,\gamma}$  is orbitally  stable in $\widetilde{W}$.
\item[$(ii)$] If $\gamma<0$, then the standing wave $e^{i\omega t}\varphi_{\omega,\gamma}$  is orbitally  unstable in $\widetilde{W}$.
\item[$(iii)$]  The standing wave $e^{i\omega t}\varphi_{\omega,\gamma}$  is orbitally  stable in $\widetilde{W}_{\rad}$.
\end{itemize}
\end{theorem}

The proof of  Theorem \ref{main} is based on the approach  established by  Grillakis, Shatah and Strauss in  \cite{GrilSha87,GrilSha90}. We prove the  well-posedness  of the Cauchy problem for NLS-log-$\delta$ equation on $\widetilde{W}$ in Section 3. For this purpose  we use the idea of the proof of  \cite[Theorem 9.3.4]{Caz03}. Namely, we  approximate the logarithmic nonlinearity by a Lipschitz continuous nonlinearities,  construct a sequence of global solutions of the regularized Cauchy problem in $C(\mathbb{R}, H^1(\mathbb{R}))$,  then we pass to the limit using standard compactness results, and finally we extract a subsequence which converges to the solution of  limiting equation \eqref{NLS_delta}. Section 4 is devoted to the proof of Theorem \ref{main}. We emphasize   that our stability approach does not use variational methods which are  standard in the study of  the stability of  standing waves for the  NLS with point defects (see \cite{AdaNoj13, AdaNoj13a, FJ, FO, ghw}). In Subsection \ref{sub3.1} we linearize NLS-log-$\delta$ equation around the peak-Gausson profile $\varphi_{\omega,\gamma}$ via the key functional $S_{\omega, \gamma}=E+(\omega+1) Q$.  As a result we  obtain two self-adjoint Schr\"odinger operators of harmonic oscillator type 
\begin{gather*}
 \mathcal L^\gamma_{1}=-\frac{d^2}{dx^2}+ \Big(|x|+\frac{\gamma}{2}\Big)^2-3,\quad  
\mathcal L^\gamma_{2}=-\frac{d^2}{dx^2}+ \Big(|x|+\frac{\gamma}{2}\Big)^2-1.
\end{gather*}
 Stability study requires investigation of the certain spectral properties of  $\mathcal L^\gamma_1$ and $\mathcal L^\gamma_2$ on the domain 
 $$
 \dom(\mathcal L^\gamma_{j})=\big\{f\in \dom(\mathcal H^\delta_\gamma): x^2f\in L^2(\mathbb R)\big\},\quad j\in\{1,2\}.
 $$
 The main difficulty is to count the number of negative eigenvalues of $\mathcal L^\gamma_{1}$. We propose two specific approaches to do this. For $\gamma>0$ we give a novel approach based on   the theory of extensions of symmetric operators of Krein-von Neumann. For   $\gamma<0$  we use  the analytic perturbation theory and the classical continuation argument based on the Riesz-projection. At the end of the Subsection \ref{sub3.3} we give the   proof of Theorem \ref{main}.

Let us also  mention that the extension theory was applied in   \cite{AN1} to investigate  the stability of standing waves of the NLS equation with $\delta'$-interaction on a star graph  $\mathcal{G}$  (see also \cite{AdaNoj15, AdaNoj14, BK})
 \begin{equation*}
i\partial_t \mathbf{U}-\mathbf{H}_\lambda^{\delta'} \mathbf{U} +|\mathbf{U}|^{p-1}\mathbf{U}=0,
\end{equation*}
  where $\mathbf{H}_\lambda^{\delta'}$ is the self-adjoint operator on $L^2(\mathcal{G})$ defined  for $\lambda\in\mathbb{R}\setminus\{0\}$ by
\begin{equation*}
\begin{split}
&(\mathbf{H}_\lambda^{\delta'} \mathbf{U})(x)=\left(-u_j''(x)\right)_{j=1}^N,\quad x\neq 0,\\
&\dom(\mathbf{H}_\lambda^{\delta'})=
\left\{\begin{array}{c}\mathbf{U}=(u_j)_{j=1}^N\in H^2(\mathcal{G}): u'_1(0)=...=u'_N(0),\\ \sum\limits_{j=1}^N  u_j(0)=\lambda u'_1(0)\end{array}\right\}.
\end{split}
\end{equation*}
 
\noindent{\bf Notation.} 
We denote by $L^2(\mathbb{R})$ the  Hilbert space equipped with the inner product $(u,v):=\re{\int\limits_{\mathbb{R}}u(x)\overline{v(x)}dx}$. Its norm is denoted by $||\cdot||_2$.
 By $H^1(\mathbb{R})$, $H^2(\mathbb{R}\setminus \{0\})=H^2(\mathbb{R}_-)\oplus H^2(\mathbb{R}_+)$ 
  we denote  Sobolev spaces.
  We denote by $\widetilde{W}$ and $\widetilde{W}_{\rad}$  the weighted Hilbert spaces $\big\{f\in H^1(\mathbb{R}): xf\in L^2(\mathbb{R})\big\}$  and $\big\{f\in H^1(\mathbb{R}): xf\in L^2(\mathbb{R}), f(x)=f(-x)\big\}$ respectively. 
  

Let $\mathcal A$ be a  densely defined symmetric operator in a Hilbert space $\mathfrak{H}$. The deficiency numbers of $\mathcal A$ are denoted by  $n_\pm(\mathcal A):=\dim\ker(\mathcal A^*\mp i\mathcal I)$, where $\mathcal I$ is the identity operator. The number of negative eigenvalues  is denoted by  $n(\mathcal A)$ (counting multiplicities). The spectrum (resp. point spectrum) and the resolvent set  of $\mathcal A$ are denoted by $\sigma(\mathcal A)$ (resp. $\sigma_p(\mathcal A)$) and $\rho(\mathcal A)$. 

The space dual to $\mathfrak{H}$ is denoted by $\mathfrak{H}'$, and $B(\widetilde W, \widetilde W')$ denotes the space of bounded operators from $\widetilde W$ to $\widetilde W'$.

\section{Previous results and basic notions}
For completeness of the exposition, below we will discuss key results on the standing waves of NLS-log equation.
First, let us give a definition of the orbital stability. The basic symmetry associated to equation \eqref{NLS_delta} is the phase-invariance (while the translation invariance  does not hold due to the defect). Thus,  the definition  of stability  takes into account only  this  type of symmetry and is formulated as follows.
\begin{definition}\label{dsta}  Let  $X$ be a Hilbert space. For $\eta>0$  let
$$
U_\eta(\varphi_{\omega,\gamma})=\Big\{v\in X : \inf_{\theta\in\mathbb R}\|v-e^{i\theta}\varphi_{\omega,\gamma}\|_{X}<\eta\Big\}.
$$
The standing wave $e^{i\omega t}\varphi_{\omega,\gamma}$ is \textit{(orbitally) stable} in $X$ if for any $\epsilon>0$ there exists $\eta>0$ such that for any $u_0\in U_\eta(\varphi_{\omega,\gamma})$, the solution $u(t)$ of \eqref{NLS_delta} with $u(0)=u_0$ satisfies $u(t)\in U_\epsilon(\varphi_{\omega,\gamma})$ for all $t\in \mathbb R$. Otherwise, $e^{i\omega t}\varphi_{\omega,\gamma}$ is said to be \textit{(orbitally) unstable} in $X$.
\end{definition}

It is interesting to note that  NLS-log equation \eqref{NLSL} possesses
standing-wave solutions  $u(t, x)=e^{i\omega t} \varphi_\omega(x)$ of the Gaussian shape 
$$
\varphi_\omega(x)=e^{\frac{\omega+n}{2}} e^{-\frac{1}{2} |x|^2}$$ for any
dimension $n$ and any frequency $\omega$ (see \cite{BM}). The orbital stability properties of the Gaussian profile $\varphi_\omega$ in the relevant class
\begin{equation}\label{W}
W(\mathbb{R}^n)=\{f\in H^1({\mathbb{R}^n}): |f|^2 \log|f|^2\in L^1({\mathbb R}^n)\}
\end{equation}  
have been studied in \cite{Caz83}. Cazenave showed that standing waves with  Gaussian profile 
 are  stable in $W(\mathbb{R}^n)$ under radial perturbations for $n\geq 2$.
The proof of this result is based on the fact that the space $H_{\rad}^1({\mathbb R^n})$ is compactly embedded  into $L^2({\mathbb R^n})$ for $n\geq 2$. Later  Cazenave and Lions in \cite[Remark II.3]{CL}  showed that such standing waves are  orbitallly stable on all $W(\mathbb{R}^n)$ for $n\geq 1$.

 We also remark  that  Angulo and Hernandez in \cite{AH} showed  (via the  variational approach)  the orbital stability of the ground states $\varphi_{\omega, \gamma}$   in the space $W(\mathbb{R})$ in the case of attractive $\delta$-interaction ($\gamma>0$).  In should be noted  that investigation of the orbital stability of $\varphi_{\omega, \gamma}$ in the case of repulsive $\delta$-interaction ($\gamma<0$) via constrained minimizer for the action or  the energy functional is not applicable  (see \cite[Remark 4.5]{AH}).
 
Recently   has been considered   NLS-log equation with an external potential $V$ satisfying specific conditions
 \begin{equation*}\label{NLSLP}
i\partial_t u+\Delta u -V(x)u+ u\log|u|^2=0.
\end{equation*}
   From the result of Ji and Szulkin in \cite{JS} it follows that there exist infinitely many profiles of standing wave  $u(x,t)=e^{i\omega t} \varphi_\omega$ (see also  \cite{SS}) for coercive $V$. Namely,  the elliptic equation
\begin{equation}\label{logpo}
-\Delta \varphi_\omega +(V(x)+\omega)\varphi_\omega=\varphi_\omega \log |\varphi_\omega|^2
\end{equation}
has infinitely many  solutions for $V\in C(\mathbb R^n, \mathbb R)$ such that  $\lim\limits_{|x|\to \infty} V(x)=+\infty$. Moreover they showed the existence of  a ground state solution (a nontrivial positive solution with least possible energy) for   bounded potential such that  $\omega+1+ V_{\infty}>0$,  in which   $V_{\infty}:=\lim\limits_{|x|\to \infty} V(x)=\sup\limits_{x\in \mathbb R^n} V(x)$,  and  $\sigma(-\Delta +V(x)+\omega+1)\subset (0, +\infty)$.
 For $V\equiv 0$ and $n\geq 3$, the authors  in \cite{DMS} showed the existence of infinitely many weak solutions to \eqref{logpo}. Also they showed that the Gaussian profile $\varphi_{-n}$ is nondegenerated, that is $\ker(L) = \Span\{\partial_{x_i} \varphi_{-n}: i=1, 2,..,N\}$, where $Lu = -\Delta u+(|x|^2-n -2)u$ is the linearized operator for $-\Delta u-n u = u\log|u|^2$ at $\varphi_{-n}$.


 The main
advantage of using the delta potential $V(x)=-\gamma\delta(x)$
is the existence of  explicit expression \eqref{gausson} for the profile $\varphi_{\omega,\gamma}$ (see Figures 1 (a) and (b)) satisfying the equation
\begin{equation}\label{eq2}
\mathcal H_\gamma^\delta\varphi+\omega\varphi-\varphi\log|\varphi|^2=0.
\end{equation}
\vskip 1cm
 \begin{tikzpicture}[scale=0.5]
\begin{axis}[
  axis x line=middle, axis y line=middle,
  ymin=0, ymax=1, xtick={0},
  xmin=-10, xmax=10, ytick={0}
]
\addplot[black,domain=-10:-2,smooth,thick] {exp{-(abs(x)+1)^2}/5};
\addplot[black,domain=2:10,smooth,thick] {exp{-(abs(x)+1)^2}/5};
\addplot[black,domain=-2:2,thick] {exp{-(abs(x)+1)^2}/5};
\end{axis}
\node[label={[xshift=2cm, yshift=-1.5cm] Figure 1(a) \text{:}\, $\varphi_{\omega, \gamma}(x)$ for
$\gamma>0$}]{};
\end{tikzpicture}
\hspace{0.7cm}
\begin{tikzpicture}[scale=0.5]
\begin{axis}[
  axis x line=middle, axis y line=middle,
  ymin=0, ymax=2, xtick={0},
  xmin=-20, xmax=20, ytick={0}
]
\addplot[black,domain=-20:-2,smooth,thick] {exp{-(abs(x/3)-1)^2}/5};
\addplot[black,domain=2:20,smooth,thick] {exp{-(abs(x/3)-1)^2}/5};
\addplot[black,domain=-2:2,thick] {exp{-(abs(x/3)-1)^2}/5};
\end{axis}
\node[label={[xshift=2cm, yshift=-1.5cm] Figure 1(b)\text{:}\, $\varphi_{\omega, \gamma}(x)$ for
$\gamma<0$}]{};
\end{tikzpicture} 
 This peak-Gausson profile is constructed from the known solution of \eqref{eq2} in the case  $\gamma=0$ on each side of the defect pasted together at $x=0$ to satisfy the continuity and the jump condition  $\varphi'(0+)-\varphi'(0-)=-\gamma\varphi(0)$ at $x=0$. 
Moreover, the following result holds. 
\begin{theorem}\label{uniq}
The set of all solutions to \eqref{eq2} is given by  $\{e^{i\theta}\varphi_{\omega,\gamma}:\,\theta\in\mathbb{R}\}$.
\end{theorem}
\noindent The proof of this theorem can be found in Appendix.

 As it was announced in Theorem \ref{main}, our stability analysis is elaborated in the specific space $\widetilde W$. To explain the choice of this space let us introduce the following two  basic functionals  associated with equation \eqref{NLS_delta}:
\begin{itemize}
\item``charge" functional $$
Q(u)=\frac{1}{2}\int_{\mathbb R} |u(x)|^2 dx,$$ 
\item``energy" functional
\begin{equation}\label{energy}
E(u)=\tfrac{1}{2}||\partial_x u||_2^2-\frac{1}{2}\int_{\mathbb R} |u(x)|^2\log|u(x)|^2 dx-\frac{\gamma}{2} |u(0)|^2.
\end{equation}
\end{itemize}
These functionals are continuosly differentiable in $W(\mathbb{R})$ defined  by \eqref{W} (see \cite{Caz83}).
 Moreover, at least formally,  $E$ is conserved by the flow of \eqref{NLS_delta}. The use  of the space $
W(\mathbb R)$ is mainly due to the fact that the  functional $E$  fails to be continuosly differentiable on  $H^1(\mathbb R)$ (see \cite{Caz83}).  As we use the approch by Grillakis, Shatah and Strauss \cite{GrilSha87,GrilSha90}, the functional $E$ needs to be twice continuously differentiable at the  $\varphi_{\omega, \gamma}$. To satisfy this condition   we propose   the  ``weighted space''  (i.e., $X$ coincides with $\widetilde{W}$ in Definition \ref{dsta})
\begin{equation}\label{tilde_W}
\widetilde{W}=H^1(\mathbb R)\cap L^2(x^2dx)=\big\{f\in H^1(\mathbb{R}): xf\in L^2(\mathbb{R})\big\}.
\end{equation}
In particular, the space $\widetilde{W}$ naturally appears in definition of the  linearization of the second derivative of $S_{\omega,\gamma}=E+(\omega+1) Q$ at $\varphi_{\omega, \gamma}$.
Note that, due to the inclusion  $\widetilde{W}\subset W(\mathbb R)$  (see Lemma \ref{incl} below), the functional $E$ is  continuously differentiable on $\widetilde{W}$.
\section{The Cauchy problem in $\widetilde{W}$}
In this section we prove the well-posedness of the Cauchy problem for \eqref{NLS_delta} in the space $\widetilde{W}$. The idea of the proof is an adaptation of the proof of \cite[Theorem 9.3.4]{Caz03}. 
The following lemma implies that $Q$ and $E$ are well-defined on $\widetilde{W}$. 
\begin{lemma}\label{incl}
Let   $W(\mathbb{R})$ and  $\widetilde W$ be the Banach spaces  defined by 
\begin{equation*}
\begin{split}
&W(\mathbb{R})=\{f\in H^1(\mathbb{R}): |f|^2 \logg|f|^2\in L^1(\mathbb R)\}, \\ 
&\widetilde{W}=\big\{f\in H^1(\mathbb{R}): xf\in L^2(\mathbb{R})\big\}.
\end{split}
\end{equation*}
Then $\widetilde W\subset W(\mathbb{R})$.
\end{lemma} 
\begin{proof}
\begin{enumerate}
\item[1)] It is easily seen that   $\widetilde W\subset L^1(\mathbb{R})$. Indeed, for $f\in \widetilde W$ and $-\infty<a<0<b<\infty$ we have
\begin{equation*}
\begin{split}
&\int\limits_{\mathbb{R}}|f|dx=\int\limits_{-\infty}^a|f|dx+\int\limits_a^b|f|dx+\int\limits^{\infty}_b|f|dx\\
&=\int\limits_{-\infty}^a|f|\cdot x\cdot\frac{1}{x}dx+\int\limits_a^b|f|dx+\int\limits^{\infty}_b|f|\cdot x\cdot\frac{1}{x}dx
\end{split}
\end{equation*}
\begin{equation*}
\begin{split}
&\leq\Big(\int\limits_{-\infty}^a (xf)^2dx\Big)^{\frac{1}{2}}\Big(\int\limits_{-\infty}^a \frac{1}{x^2}dx\Big)^{\frac{1}{2}}+(b-a)\sup\limits_{[a,b]}|f|\\&+\Big(\int\limits^{\infty}_b (xf)^2dx\Big)^{\frac{1}{2}}\Big(\int\limits^{\infty}_b \frac{1}{x^2}dx\Big)^{\frac{1}{2}}< \infty.
\end{split}
\end{equation*}
\item[2)] Let again $f\in \widetilde W$, then
\begin{equation}\label{f}
\int\limits_{\mathbb{R}}|f|^2|\log|f||dx=\int\limits_{\{x\in \mathbb{R}: |f|< 1\}}|f|^2|\log|f||dx+\int\limits_{\{x\in \mathbb{R}: |f|\geq 1\}}|f|^2|\log|f||dx.
\end{equation}
Note also  that
\begin{equation}\label{log}
|\log|f||<\frac{1}{|f|}\quad \text{for}\quad |f|<1,\qquad \text{and}\qquad |\log|f||<|f|\quad \text{for}\quad |f|\geq 1.
\end{equation} 
Since $f\in H^1(\mathbb{R})$, there exists $c>0$ such that $|f|<1$ for $\mathbb{R}\setminus [-c,c]$.
 Thus, from \eqref{f},  \eqref{log},  and the inclusion $\widetilde W\subset L^1(\mathbb{R})$ we get
 \begin{equation*}
 \begin{split}
 &\int\limits_{\mathbb{R}}|f|^2|\log|f||dx\leq \int\limits_{\{ x\in \mathbb{R}: |f|< 1\}}|f|dx+\int\limits_{\{ x\in \mathbb{R}: |f|\geq 1\}}|f|^3dx \\
 &\leq\int\limits_{\{x\in \mathbb{R}: |f|< 1\}}|f|dx+2c\sup\limits_{[-c,c]}|f|^3<\infty.
 \end{split}
\end{equation*}  
The assertion is proved.
  \end{enumerate}
\end{proof}

The global well-posedness property of the Cauchy problem for \eqref{NLS_delta} is ensured by the following theorem. 
\begin{theorem}\label{exist}
If $u_0\in \widetilde{W}$, there is a unique solution $u(t)$ of \eqref{NLS_delta} such that $u(t)\in C(\mathbb R, \widetilde{W})$ and $u(0,x)=u_0$. Furthermore, the conservation of energy and charge hold, i.e., for any $t\in \mathbb R$, we have
\begin{equation*}
E(u(t))=E(u_0),\quad Q(u(t))=Q(u_0).
\end{equation*}
Moreover, if an initial data $u_0$ is even, the solution $u(t)$ is also even.
\end{theorem}
\begin{proof}
Generally we use an approach proposed in \cite{Caz80} with few natural modifications.
The proof can be divided into three parts.
 \begin{enumerate}
 \item[1)]  We introduce  the "reduced" Cauchy problem
 \begin{equation}\label{reduced}
 \begin{cases}
i\partial_tu_n-\mathcal H^\gamma_\delta u_n+u_nf_n(|u_n|^2)=0,\\
u_n(0)=u_0.
 \end{cases}
 \end{equation}
 Here $f_n(s)=\inf\{n, \sup\{-n, f(s)\}\}$ with $f(s)=\log s$, $s>0$. We define $F_n(s)=\int\limits_0^sf_n(\sigma)d\sigma$. By Theorem 3.3.1 in \cite{Caz03}, we imply that for any $u_0\in  H^1(\mathbb{R})$ there exists  unique global solution $u_n$ of \eqref{reduced} such that $u_n\in C(\mathbb{R}, H^1(\mathbb{R}))$ and $u_n(0)=u_0$. Moreover, the conservation of charge and energy hold, i.e., for all $t$
 \begin{equation*}
 \begin{split}
 &||u_n(t)||_2=||u_0||_2,\quad E_n(u_n(t))=E_n(u_0),\\
 &E_n(u(t))=\tfrac{1}{2}||\partial_x u(t)||_2^2-\tfrac{1}{2}\int_{\mathbb R} F_n(|u(t, x)|^2) dx-\tfrac{\gamma}{2} |u(t, 0)|^2.
 \end{split}
 \end{equation*}
Indeed, we may  check the assumptions of Theorem 3.3.1 in \cite{Caz03}. Note that  $f_n$ is Lipschitz continuous from $\mathbb{R}_+$ to $\mathbb{R}$.
We also notice  that $\mathcal H^\gamma_\delta$ defined in \eqref{H_delta} satisfies $\mathcal H_\delta^\gamma\geq -m$, where $m = \gamma^2/4$ if $\gamma > 0$, and $m = 0$ if $\gamma < 0$. Thus, $A =-\mathcal H^\gamma_\delta-m $ is the self-adjoint negative operator in $X = L^2(\mathbb{R})$ on the  domain $\dom(A) = \dom(\mathcal H_\delta^\gamma)$. Moreover, in our case the  norm
$$
||v||^2_{X_A}=||v'||_2^2+(m+1)||v||_2^2-\gamma|v(0)|^2
$$
is  equivalent to the usual $H^1(\mathbb{R})$-norm.

\item[2)] The second step is analogous to Lemma 2.3.5 in \cite{Caz80}. In particular, it can be shown that there exists solution $u$ of \eqref{NLS_delta} in the sense of distributions (which appears to be weak-$*$ limit of solutions $u_{n_k}$ of Cauchy problem \eqref{reduced})  such that conservation of charge holds.  Further, conservation of energy $E(u)$  defined by \eqref{energy} follows from its monotonicity. Thus, inclusion $u\in C(\mathbb{R}, H^1(\mathbb{R}))$ follows from conservation laws.
 \item[3)] The last step is to show that the inclusion $xu_0\in L^2(\mathbb{R})$ implies the inclusion $xu\in L^2(\mathbb{R})$. The proof of this fact repeats one of Lemma 7.6.2 from \cite{Caz98}. 
 \end{enumerate}
\end{proof}
\begin{remark}
 For the  completeness of the exposition we remark that for $\gamma>0$ the unitary group $G_\gamma(t)=e^{-it\mathcal H^\gamma_\delta}$ associated to  equation \eqref{reduced} (or equivalently  to \eqref{NLS_delta}) is given explicitly  by the formula (see \cite{AlbBrz95, GavSch86})
\begin{equation*}\label{pospro1}
G_\gamma(t) \phi(x)= e^{it \Delta} (\phi\ast \tau_\gamma) (x) \chi^0_{+} + \Big[e^{it \Delta} \phi(x) + e^{it \Delta} (\phi\ast \rho_\gamma) (-x) \Big ]\chi^0_{-},
\end{equation*}
where
$$
\rho_\gamma(x)=-\frac{\gamma}{2} e^{\frac{\gamma}{2} x}\chi^0_{-},\;\; \tau_\gamma(x)=\delta (x)+ \rho_\gamma(x).
$$
Here  $\chi^0_{+}$ and $\chi^0_{-}$  denote  the characteristic functions of $[0,+\infty)$ and $(-\infty, 0]$ respectively.
\end{remark}

\section{Proof of the main result}
In this Section we  prove  Theorem \ref{main}. Initially we define key functional $S_{\omega,\gamma}$ associated with NLS-log-$\delta$ equation. Next we establish the relation  between the second variation  of $S_{\omega,\gamma}$ and the self-adjoint operators $\mathcal L^\gamma_{2}$ and $\mathcal L^\gamma_{1}$. Verifying the spectral properties of $\mathcal L^\gamma_{2}$ and $\mathcal L^\gamma_{1}$,  we arrive at the assertions of Theorem \ref{main}. In our analysis we follow some  ideas from  \cite{CozFuk08}.

\subsection{Linearization of NLS-log-$\delta$ equation}\label{sub3.1}
We start  introducing the key functional $S_{\omega,\gamma}=E+(\omega+1) Q.$ 
It can be easily verified that the profile $\varphi_{\omega,\gamma}$  is a critical point of $S_{\omega,\gamma}$.
 Indeed, for $u,v\in\widetilde{W}$,
 \begin{equation*}
 \begin{split}
& S_{\omega,\gamma}'(u)v=\frac{d}{dt}S_{\omega,\gamma}(u+tv)|_{t=0}\\&=\re\left[\int\limits_{\mathbb{R}}u'\overline{v'}dx-\int\limits_{\mathbb{R}}u\overline{v}(\log|u|^2-\omega)dx-\gamma u(0)\overline{v(0)}\right].
 \end{split}
 \end{equation*}
Since $\varphi_{\omega,\gamma}$  satisfies \eqref{eq2}, $S_{\omega,\gamma}'(\varphi_{\omega,\gamma})=0$. 

In the approach by  \cite{GrilSha90} crucial role is played by spectral properties of the linear operator associated with the second variation of $S_{\omega,\gamma}$ calculated  at $\varphi_{\omega,\gamma}$. To express $S_{\omega,\gamma}'' (\varphi_{\omega,\gamma})$ it is convenient to split $u, v\in \widetilde{W}$ into real and imaginary parts: $u=u_1+iu_2,\,\,v=v_1+iv_2$. Then  we get 
\begin{equation*}\label{q_form}
\begin{split}
&S_{\omega,\gamma}''(\varphi_{\omega,\gamma})(u, v)=\int\limits_{\mathbb{R}}u'_1v'_1dx-\int\limits_{\mathbb{R}}u_1 v_1(\log|\varphi_{\omega,\gamma}|^2-\omega+2)dx -\gamma u_1(0)v_1(0) \\
&+\int\limits_{\mathbb{R}}u'_2v'_2dx-\int\limits_{\mathbb{R}}u_2 v_2( \log|\varphi_{\omega,\gamma}|^2-\omega)dx -\gamma u_2(0)v_2(0)
\\&=\int\limits_{\mathbb{R}}u'_1v'_1dx+\int\limits_{\mathbb{R}}u_1 v_1 \left(\left(|x|+\tfrac{\gamma}{2}\right)^2-3\right)dx-\gamma u_1(0)v_1(0)\\&+\int\limits_{\mathbb{R}}u'_2v'_2dx+\int\limits_{\mathbb{R}}u_2 v_2 \left(\left(|x|+\tfrac{\gamma}{2}\right)^2-1\right)dx-\gamma u_2(0)v_2(0).
\end{split}
\end{equation*}
Therefore, $S_{\omega,\gamma}''(\varphi_{\omega,\gamma})(u, v)$ can be formally rewritten as 
\begin{equation}\label{SBB}
S_{\omega,\gamma}''(\varphi_{\omega,\gamma})(u, v)=B^\gamma_{1}(u_1,v_1)+B^\gamma_{2}(u_2,v_2),
\end{equation}
where 
\begin{equation}\label{spec14}
\begin{split}
&B^\gamma_{1}(f,g)=\int\limits_{\mathbb{R}}f'g'dx+\int\limits_{\mathbb{R}}f g \left(\left(|x|+\tfrac{\gamma}{2}\right)^2-3\right)dx-\gamma f(0)g(0),\\
&B^\gamma_{2}(f,g)=\int\limits_{\mathbb{R}}f'g'dx+\int\limits_{\mathbb{R}}f g \left(\left(|x|+\tfrac{\gamma}{2}\right)^2-1\right)dx-\gamma f(0)g(0),
\end{split}
\end{equation} 
and $\dom(B^\gamma_{j})=\widetilde{W}\times \widetilde{W}, j\in\{1,2\}$.
Note that the forms $B^\gamma_{j}, j\in\{1,2\},$ are bilinear 
 bounded from below and closed. Therefore, by  the First Representation Theorem (see  \cite[Chapter VI, Section 2.1]{kato}), they define operators $\mathcal L^\gamma_{1}$ and $\mathcal L^\gamma_{2}$ such that for $ j\in\{1,2\}$
 \begin{equation}\label{opera}
 \begin{split}
 &\dom(\mathcal L^\gamma_{j})=\{v\in \widetilde W: \exists w\in   L^2(\mathbb{R})\; s.t.\; \forall z\in \widetilde W, \;B^{\gamma}_j(v,z)=(w,z) \},\\
& \mathcal L^\gamma_{j}v=w.
 \end{split}
\end{equation}
In the following theorem we  describe the  operators $\mathcal L^\gamma_{1}$ and $\mathcal L^\gamma_{2}$ in more explicit form.  We  show  that they are basically the harmonic oscillator operators with $\delta$-interaction. 
\begin{theorem}\label{repres}
The operators $\mathcal L^\gamma_{1}$ and $\mathcal L^\gamma_{2}$ determined in (\ref{opera}) are given by 
\begin{equation*}
 \begin{split}
 \mathcal L^\gamma_{1}=-\frac{d^2}{dx^2}+ \Big(|x|+\frac{\gamma}{2}\Big)^2-3, \quad
 \mathcal L^\gamma_{2}=-\frac{d^2}{dx^2}+ \Big(|x|+\frac{\gamma}{2}\Big)^2-1 
 \end{split}
\end{equation*} 
on the domain $D_\gamma:=\{f\in\dom(\mathcal{H}_\gamma^\delta):\, x^2f\in L^2(\mathbb{R})\}$.
\end{theorem}
\begin{proof}
Since the proof for $\mathcal L^\gamma_{2}$ is similar to the one for $\mathcal L^\gamma_{1}$, we  deal with $\mathcal L^\gamma_{1}$. Let 
$B^{\gamma}_1= B^{\gamma} + B_1$, where $B^{\gamma}:  H^1(\mathbb{R})\times  H^1(\mathbb{R})\to \mathbb R$ and $ B_{1}: \widetilde{W}\times  \widetilde{W}\to \mathbb R$ are  defined by
\begin{equation*}
B^{\gamma}(u,v)=(u',v')-\gamma u(0)v(0),\quad B_{1}(u,v)=(V^\gamma_1 u,v),
\end{equation*}
and $V^\gamma_1 (x)=\Big(|x|+\frac{\gamma}{2}\Big)^2-3$. We denote by $\mathcal L^\gamma$ (resp. $\mathcal L_{1}$) the self-adjoint operator on 
$L^2(\mathbb{R})$  associated (by the First Representation Theorem) with $B^{\gamma}$ (resp. $B_1$). Thus,  
\begin{equation*}
  \begin{split}
&\dom(\mathcal L^\gamma)=\{v\in H^1(\mathbb{R}): \exists w\in   L^2(\mathbb{R})\; s.t.\; \forall z\in H^1(\mathbb{R}), \;B^{\gamma}(v,z)=(w,z) \},\\ &\mathcal L^\gamma v=w.
\end{split}
\end{equation*}
   We claim  that $\mathcal L^\gamma$ is a self-adjoint extension of the following symmetric operator 
\begin{equation*}\label{L0}
 \mathcal L^0=-\frac{d^2}{dx^2}, \quad \dom(\mathcal L^0)=\{ f\in  H^2(\mathbb R):  f(0)=0 \}.
\end{equation*} 
Indeed, let $v\in \dom(\mathcal L^0)$ and   $w= -v''\in L^2(\mathbb{R})$. Then for every $z\in  H^1(\mathbb{R})$ we have $B^{\gamma}(v,z)=(w,z)$. 
Thus,  $v\in \dom(\mathcal L^\gamma)$ and $\mathcal L^\gamma v=w=-v''$. Hence, $\mathcal L^0\subset \mathcal L^\gamma$, which yields the claim. Therefore, by \cite[Theorem 3.1.1]{AlbGes05}, there exists $\beta\in \mathbb R$ such that $\mathcal L^\gamma=-\Delta_{\beta}$, where
\begin{equation*}
 \begin{split}
 &-\Delta_{\beta}=-\frac{d^2}{dx^2}, \\
& \dom(-\Delta_{\beta})=\{f\in H^1(\mathbb R)\cap  H^2(\mathbb R\setminus\{0\}):\; 
 f'(0+)-f'(0-)=\beta f(0)\}.
 \end{split}
\end{equation*} 
 Next we shall prove that $\beta=-\gamma$. Take $v\in \dom(\mathcal L^\gamma)$ with $v(0)\neq 0$, then 
we obtain $$
(\mathcal L^\gamma v,v)=(v'(0+)-v'(0-))v(0) +\|v'\|_2^2=\|v'\|_2^2+\beta |v(0)|^2,$$
which should be equal to $B^{\gamma}(v,v)=\|v'\|_2^2-\gamma |v(0)|^2$.
Therefore, $\beta=-\gamma$.

Again, by	 the First Representation Theorem, 
\begin{equation*}
\begin{split}
  &\dom(\mathcal L_{1})=\{v\in \widetilde{W}: \exists w\in   L^2(\mathbb{R})\; s.t.\; \forall z\in \widetilde{W}, \;B_1(v,z)=(w,z)\},\\
  &\mathcal L_{1}v=w.
  \end{split}
  \end{equation*}
Note that $\mathcal L_1$ is the
 self-adjoint extension of the following  multiplication operator 
\begin{equation*}
  \mathcal L_{0,1} f =V^\gamma_1 f , \quad
 \dom(\mathcal L_{0,1})=\{f\in  H^2(\mathbb{R}):  V^\gamma_1 f\in L^2(\mathbb{R})\}.
 \end{equation*} 
 Indeed, for $v\in \dom(\mathcal L_{0,1})$ we have $v\in \widetilde{W}$, and we define $w= V^\gamma_1 v\in L^2(\mathbb{R})$. Then for every $z\in \widetilde{W}$ we get $B_1(v,z)=(w,z)$. Thus, 
$v\in \dom(\mathcal L_1)$ and $\mathcal L_1v=w=V^\gamma_1 v$. Hence, $\mathcal L_{0,1}\subseteq \mathcal L_1$.  Since $\mathcal L_{0,1}$ is self-adjoint, $\mathcal L_1= \mathcal L_{0,1}$.
The Theorem is proved.
\end{proof}
\begin{remark}
\begin{itemize}
\item[$(i)$] We mention that  $\widetilde{W}$ coincides with the natural domain of the bilinear forms $B_1^\gamma$ and $B_2^\gamma$ which additionally justifies the choice of the space $\widetilde W$ for investigation of the orbital stability.
 \item[$(ii)$] It's worth mentioning that   the operators $\mathcal L^\gamma_{1}$ and $\mathcal L^\gamma_{2}$ coincide up to a constant term, namely, $\mathcal L^\gamma_{1}=\mathcal L^\gamma_{2}-2$.
 This is a special feature
of the logarithmic nonlinearity  that clarifies and simplifies the spectral
analysis implemented in the next Subsection. 
\item[$(iii)$] We remark that it is also possible to propose  an alternative   proof of Theorem \ref{repres} avoiding the decomposition of the forms $B^\gamma_j$ into the the sum of two forms, though it  will require more extensive proof. Indeed, the self-adjoint operator $\mathcal L^\gamma_{1}$ associated with the form $B^\gamma_1$ is a self-adjoint extension of the symmetric operator  $\mathcal L_0$ defined in \eqref{L_0}. By \cite[Chapter IV,\S 14, Theorems 7 and 8]{Nai67}, we get 
\begin{equation}\label{D_g}
\dom(\mathcal L^\gamma_{1})=\left\{\begin{array}{c} f: 
f=f_0+cf_i+ce^{i\theta}f_{-i},\\
\,f_0\in\dom(\mathcal L_0),c\in\mathbb{C}, \theta\in[0,2\pi)\end{array}\right\},
\end{equation}
where $f_{\pm i}$ are deficiency vectors, namely, $\ker(\mathcal L_0^*\mp i\mathcal I)=\Span\{f_{\pm i}\}$. The deficiency vector $f_{i}$ has the form  (we note that $f_{-i}=\overline{f_i}$)
$$
f_{i}(x)=\begin{cases}
C_1U\Big(-\tfrac{3+i}{2}, \sqrt{2}(x+\tfrac{|\gamma|}{2})\Big),\,\, x>0,\\
C_2U\Big(-\tfrac{3+i}{2}, \sqrt{2}(x-\tfrac{|\gamma|}{2})\Big),\,\, x<0,
\end{cases}$$
where   $C_1, C_2$  are fixed constants that  guarantee  continuity of $f_i$ at $x=0$.  The function $U\Big(-\tfrac{3+i}{2},\, \cdot\Big)$ was found  reducing the equation 
$$-f''(x)+(|x|+\tfrac{\gamma}{2})^2f(x)-(3+i)f(x)=0,$$
via change of variables to the Weber equation (see (19.1.2) in \cite{AbrSte83})
$$g''(z)-(\tfrac 1{4}z^2-\tfrac{3+i}{2})g(z)=0.$$
This equation has the  solution $U(a,z)$ (with $a=-\tfrac{3+i}{2}$) such that $\lim \limits_{|z|\rightarrow \infty}U(a,z)=0$ (see (19.8.1) in \cite{AbrSte83} and also \cite[Chapter 6]{BeaWon10}). In particular, $U(a,z)$ is given by  (see (19.3.1), (19.3.3), (19.3.4) in \cite{AbrSte83})
$$U(a,z)={\frac  {1}{2^{\xi }{\sqrt  {\pi }}}}\left[\cos(\xi \pi )\Gamma (1/2-\xi )\,y_{1}(a,z)-{\sqrt  {2}}\sin(\xi \pi )\Gamma (1-\xi )\,y_{2}(a,z)\right],$$
where $\xi ={\frac  {1}{2}}a+{\frac  {1}{4}}$ and 
\begin{equation*}
\begin{split}
&y_{1}(a,z)=\exp(-z^{2}/4) _{1}F_{1}\left({\tfrac  12}a+{\tfrac  14};{\tfrac  12};{\frac  {z^{2}}{2}}\right),\\
&y_{2}(a,z)=z\exp(-z^{2}/4) _{1}F_{1}\left({\tfrac  12}a+{\tfrac  34};{\tfrac  32};{\frac  {z^{2}}{2}}\right),
\end{split}
\end{equation*}
in which  $_{1}F_{1}(\cdot;\cdot;\cdot)$ is the confluent hypergeometric function (see \cite[Chapter 13]{AbrSte83}). The function $U(a,z)$ is called \textit{parabolic cylinder function}.
Using the definition of $y_{1}(a,z)$ and $y_{2}(a,z)$, it can be shown (after laborious  calculations) 
that the  set \eqref{D_g} coincides with $D_\gamma$.
\end{itemize}
\end{remark}

Next, we consider the  form  $S_{\omega,\gamma}''(\varphi_{\omega,\gamma}):\widetilde{W}\times\widetilde{W}\rightarrow \mathbb{C}$ as a linear operator $\mathcal H_{\omega, \gamma}: \widetilde{W}\rightarrow \widetilde{W}'$. 
Our main stability result follows from the next theorem (see \cite[Instability Theorem and Stability Theorem]{GrilSha90}).
\begin{theorem}\label{st0} 
Let $\gamma \neq 0$  and  
\begin{equation*}\label{st}
p_\gamma(\omega_0)=
\begin{cases}
\begin{aligned}
&1, \quad  {\text{if}} \;\; \partial_\omega\|\varphi_{\omega,\gamma}\|_2^2>0,\;\;at\;\; \omega=\omega_0,\\
&0, \quad  {\text{if}} \;\; \partial_\omega\|\varphi_{\omega,\gamma}\|_2^2<0,\;\;at\;\; \omega=\omega_0.
\end{aligned}
\end{cases}
\end{equation*}
Then the following assertions hold.
\begin{itemize}
\item[$(i)$] If $n(\mathcal H_{\omega_0, \gamma})=p_\gamma(\omega_0)$, then the standing wave $e^{i\omega t}\varphi_{\omega,\gamma}$  is orbitally  stable in  $\widetilde{W}$. 
\item[$(ii)$] If $n(\mathcal H_{\omega_0, \gamma})-p_\gamma(\omega_0)$ is odd, then the standing wave $e^{i\omega t}\varphi_{\omega,\gamma}$  is orbitally  unstable in   $\widetilde{W}$. 
\end{itemize}
\end{theorem}
\begin{remark}
Analogous result holds for the case of the space $\widetilde{W}_{\rad}$. 
\end{remark}
Due to  \cite{GrilSha90}, the proof of this theorem requires verification of \textit{Assumptions 1,2,3}. 

  \textit{ Assumption 1 and Assumption 2} obviously hold: 
\begin{itemize}
\item well-posedness of equation \eqref{NLS_delta}  (Theorem \ref{exist}),
\item  the existence of a smooth curve of peak standing-wave $\omega \to \varphi_{\omega, \gamma}$ (see (\ref{gausson})).
\end{itemize}  
 Checking  \textit{Assumption 3} in \cite{GrilSha90}  is equivalent to  the following Theorem.
\begin{theorem}\label{ass3} Let $\gamma\neq 0$, then for any $\omega\in\mathbb{R}$ the following assertions hold.
\begin{itemize}
\item[$(i)$]  The operator $\mathcal H_{\omega, \gamma}$
 has only  a finite number of  negative eigenvalues.
\item[$(ii)$] The kernel of $\mathcal H_{\omega, \gamma}$  coincides with  $\Span\{i\varphi_{\omega,\gamma}\}$.
\item[$(iii)$] The rest of the spectrum of $\mathcal H_{\omega, \gamma}$ is positive and bounded away from zero.
\end{itemize}
\end{theorem}
\noindent This Theorem will be proved  below. From  \eqref{SBB} we can define formally 
 \begin{equation}\label{spec0}
 \mathcal H_{\omega, \gamma} u= \mathcal L^\gamma_{1}u_1+i \mathcal L^\gamma_{2}u_2,
\end{equation}
where $u_1=\re(u),\,u_2=\im(u)$.
In connection with  Theorem \ref{st0} and Theorem \ref{ass3} our aim is to  investigate   the
following three spectral conditions associated to  $\mathcal L_1^\gamma$ and  $\mathcal L_2^\gamma$ :
\begin{enumerate}
\item[$\bullet$] the  operator $\mathcal L_2^\gamma$ has $\ker(\mathcal L^\gamma_{2})=\Span\{\varphi_{\omega, \gamma}\}$ and $\inf (\sigma(\mathcal L^\gamma_{2})\setminus\{0\})>\varepsilon>0$;
\item[$\bullet$] the  operator $\mathcal L^\gamma_{1}$ 
has a trivial kernel for all $\gamma\in \mathbb R\setminus\{0\}$, and  $\inf (\sigma(\mathcal L^\gamma_{1})\cap\mathbb{R}_+)>\varepsilon>0$, while $\sigma(\mathcal L^\gamma_{1})\cap\mathbb{R}_{-}=\{\lambda_k\}_{k=1}^n,$ where $n<\infty$;
\item[$\bullet$] the number of negative eigenvalues of the operator $\mathcal L^\gamma_{1}$.
\end{enumerate}
These three conditions will be studied in the next Subsection. The main difficulty  is to count the number of negative eigenvalues of $\mathcal L^\gamma_{1}$. We use two specific approaches to do this. For $\gamma>0$ we apply exclusively the theory of extensions of symmetric operators. In the case  $\gamma<0$, we  consider  $\mathcal L^\gamma_{1}$  as a  real-holomorphic perturbation of the one-dimensional harmonic oscillator operator
\begin{equation}\label{osci}
 \mathcal L^0_{1} =-\frac{d^2}{dx^2}+ x^2-3,\quad \dom(\mathcal L^0_{1})=\{f\in H^2(\mathbb{R}): x^2f\in L^2(\mathbb{R})\}. 
\end{equation} 
Using the perturbation theory, we claim that the point spectrum of $\mathcal L^\gamma_{1}$  depends holomorphically on the spectrum of  $\mathcal L^0_{1}$. In particular,  for $\gamma<0$  we show the equality $n(\mathcal L^\gamma_{1})=2$, while  for $\gamma>0$  we obtain $n(\mathcal L^\gamma_{1})=1$. Moreover, we show that  $n(\mathcal L^\gamma_{1})=1$ in the space $\widetilde{W}_{\rad}$ for any $\gamma\in\mathbb{R}\setminus \{0\}$.

\subsection{Spectral properties of $\mathcal L^{\gamma}_1$ and $\mathcal L^{\gamma}_2$}\label{sub3.3}
Below we discuss the spectral properties of $\mathcal L^{\gamma}_j,\,j\in\{1,2\}$. Let us make few general observations. First, since  
$$
\lim\limits_{|x|\to +\infty} \left(|x|+\frac{\gamma}{2}\right)^2=+\infty,
$$
 the operators $\mathcal L^\gamma_j, \,\,j\in\{1,2\}$,  have a discrete spectrum, $\sigma(\mathcal L^\gamma_j)=\sigma_p(\mathcal L^\gamma_j)=\{\lambda_{k}^j\}_{k\in\mathbb{N}}$ (see \cite[Chapter II]{BerShu91}). In particular, we have the following distribution of the eigenvalues
$$
\lambda^j_0<\lambda^j_1<\cdot\cdot\cdot<\lambda^j_k<\cdot\cdot\cdot,
$$
with $\lambda_k^j\to +\infty$ as $k\to +\infty$. Due to semi-boundedness of $V^1_\gamma=(|x|+\frac{\gamma}{2})^2-3$ and $V^2_\gamma=(|x|+\frac{\gamma}{2})^2-1$,  we obtain that any nontrivial solution of the equation
$$
\mathcal L^\gamma_j v=\lambda_k^j v, \qquad v\in \dom(\mathcal L^\gamma_j),
$$
 is unique up to a constant factor (see \cite{BerShu91}). Therefore, each eigenvalue $\lambda^j_{k}$ is simple.  
Moreover, the following Proposition holds.

\begin{proposition}
Let $\gamma \in \mathbb R\setminus\{0\}$.  Then $\ker(\mathcal H_{\omega, \gamma})=\Span\{i\varphi_{\omega, \gamma}\}$.
\end{proposition}
\begin{proof}
 Since $\varphi_{\omega, \gamma}\in D_\gamma$ and  $\mathcal L^\gamma_2 \varphi_{\omega, \gamma}=0$, we obtain  immediately  $\ker(\mathcal L^\gamma_2)=\Span\{\varphi_{\omega, \gamma}\}$.
Now,  suppose that $u\in\ker(\mathcal L^\gamma_{1})$ and $u\neq 0$. It means that $u\in H^1(\mathbb R)\cap H^2(\mathbb R\setminus\{0\})$  and  
 \begin{gather}\label{ker_L1} 
 -u''+((|x|+\tfrac{\gamma}{2})^2-3)u=0,\,\,x\neq 0,\\\label{ker_L1'}
   u'(0+)-u'(0-)=-\gamma u(0).
 \end{gather} 
 Consider  \eqref{ker_L1}  on $(0, \infty)$. Then, the fact that $\ker(\mathcal L^\gamma_2)=\Span\{\varphi_{\omega, \gamma}\}$ implies 
 \begin{equation}\label{ker_L2}
 -\varphi_{\omega, \gamma}''+((|x|+\tfrac{\gamma}{2})^2-1)\varphi_{\omega, \gamma}=0\quad \text{on}\quad (0, \infty).
 \end{equation}
 Differentiating \eqref{ker_L2}, we obtain that $\varphi_{\omega, \gamma}'$ satisfies \eqref{ker_L1} on $(0, \infty)$. Since we look for $L^2(\mathbb{R})$-solution, every solution of  \eqref{ker_L1} in $(0, \infty)$ is of the form $\mu\varphi_{\omega, \gamma}', \mu\in\mathbb{R}$ \cite[Chapter II]{BerShu91}. Analogously, every solution in    $(-\infty,0)$ is given by $\nu\varphi_{\omega, \gamma}', \nu\in\mathbb{R}$.  Thus, the solution $u$ of \eqref{ker_L1}-\eqref{ker_L1'} has the form 
 $$
 u=\left\{\begin{array}{lc}-\mu\varphi_{\omega, \gamma}',\,\,\;\; x\in(-\infty,0),\\
\mu\varphi_{\omega, \gamma}',\,\,\;\;\;\;x\in(0,\infty).\\
\end{array}\right.
$$ 
Since $u\in H^1(\mathbb{R})$ and $\varphi_{\omega, \gamma}$ satisfies condition    \eqref{ker_L1'}, we get
\begin{equation*}\label{u0}
u(0)=-\mu\varphi_{\omega, \gamma}'(0-)=\mu\varphi_{\omega, \gamma}'(0+)=-\tfrac{\mu}{2}\gamma\varphi_{\omega, \gamma}(0).
\end{equation*}
On the other hand, the fact that $\varphi_{\omega, \gamma}$ satisfies \eqref{ker_L2} implies 
\begin{equation*}\label{u'0}
u'(0\pm)=\pm\mu\varphi_{\omega, \gamma}''(0\pm)=\pm\mu(\tfrac{\gamma^2}{4}-1)\varphi_{\omega, \gamma}(0).
\end{equation*}
Finally, using \eqref{ker_L1'}, we arrive at
$$2\mu(\tfrac{\gamma^2}{4}-1)\varphi_{\omega, \gamma}(0)=\tfrac{\mu}{2}\gamma^2\varphi_{\omega, \gamma}(0).$$
This is a contradiction, therefore $\mu=0$ and $u\equiv 0.$ The equality  $\ker(\mathcal H_{\omega, \gamma})=\Span\{i\varphi_{\omega, \gamma}\}$ follows from   \eqref{spec0}.
 \end{proof}
\noindent Summarizing the above facts we arrive at   \textit{the proof of Lemma \ref{ass3}}.
  
\noindent The following result implies non-negativity of the operator $\mathcal L^\gamma_2$.
  \begin{proposition}\label{nL2}
Let $\gamma \in \mathbb R\setminus\{0\}$. Then  $n(\mathcal L^\gamma_2)=0$. 
  \end{proposition}
    The proof of Proposition \ref{nL2}  follows from positivity of $\varphi_{\omega,\gamma}$ and  the following generalization of the classical Sturm oscillation theorem to the case of point interaction  (see \cite{BerShu91}).
\begin{lemma}\label{Oscill}
Let $V(x)$ be a real-valued continuous function on $\mathbb{R}$.
Let also $\varphi_1, \varphi_2\in L^2(\mathbb{R})$ be eigenfunctions of the operator
\begin{equation*}
L_V=-\frac{d^2}{dx^2}+V(x),\quad
\dom(L_V)=\left\{f\in \dom(\mathcal{H}_\gamma^\delta): L_Vf\in L^2(\mathbb{R})\right\},
\end{equation*}
corresponding to the eigenvalues $\lambda_1<\lambda_2$
 respectively. Suppose that $n_1$  and $n_2$ are the number of zeroes of $\varphi_1, \varphi_2$ respectively. Then $n_2>n_1$.
\end{lemma}
\begin{proof}  Suppose that $\varphi_1(a)=\varphi_1(b)=0$ and $-\infty< a<0<b\leq\infty$, besides  $\varphi_1(\infty)=0$ is understood in the sense of limit.   Let also $\varphi_1>0$ in $(a, b)$. Then $\varphi'_1(a)>0$ and $\varphi'_1(b)\leq 0$.  The  "equality" $\varphi'_1(b)=0$ takes place only if $b=\infty$ since $\varphi_1\in H^2(0,\infty)$.
Suppose that $\varphi_2$  has no zeros in $(a,b)$   and $\varphi_2>0$ in $(a,b)$.
Using the fact that $\varphi_1, \varphi_2$ are eigenfunctions of $L_V$, we arrive at
\begin{equation}\label{eigen}
\begin{split}
&0=\int\limits^b_a(\varphi_1\varphi''_2-\varphi''_1\varphi_2) dx+\int\limits^b_a(\lambda_2-\lambda_1)\varphi_1\varphi_2 dx \\&=\int\limits^{-0}_a\frac{d}{dx}(\varphi_1\varphi'_2-\varphi'_1\varphi_2) dx+\int\limits^b_{+0}\frac{d}{dx}(\varphi_1\varphi'_2-\varphi'_1\varphi_2) dx+ \int\limits^b_a(\lambda_2-\lambda_1)\varphi_1\varphi_2 dx\\
&=\left[\varphi_1\varphi'_2-\varphi'_1\varphi_2\right]^b_a+\left[\varphi'_1\varphi_2-\varphi_1\varphi'_2\right]^{+0}_{-0}+\int\limits^b_a(\lambda_2-\lambda_1)\varphi_1\varphi_2 dx.
\end{split}
\end{equation}
Since $\varphi_1, \varphi_2\in \dom(\mathcal{H}_\gamma^\delta)$, we get  $\left[\varphi'_1\varphi_2-\varphi_1\varphi'_2\right]^{+0}_{-0}=0$. Therefore, from \eqref{eigen} and initial assumptions it  easily follows that
\[0>\left[\varphi_1\varphi'_2-\varphi'_1\varphi_2\right]^b_a=\varphi'_1(a)\varphi_2(a)-\varphi'_1(b)\varphi_2(b)>0,\]
which is a contradiction. Thus, $\varphi_2$ has at least one zero in $(a,b)$. Analogously, we can prove that  there  exists $\xi\in(-\infty, a]$ such that $\varphi_2(\xi)=0$. 	
Thereby,  between two  finite zeroes of $\varphi_1$ there exists a
zero of $\varphi_2$, and between the last finite zero of $\varphi_1$ and $\infty$ (between the first finite zero of  $\varphi_1$ and $-\infty$ respectively) there is  at least one zero of $\varphi_2$. The proof is completed.
\end{proof}

 \begin{remark}\label{morse}
Note that from $\ker(\mathcal L^\gamma_{2})=\Span\{\varphi_{\omega, \gamma}\}$ and $\inf (\sigma(\mathcal L^\gamma_{2})\setminus\{0\})>\varepsilon>0$ it follows that $n(\mathcal H_{\omega, \gamma})=n(\mathcal L^\gamma_{1})$.
\end{remark}
\vspace{0.5cm} 
\noindent{\bf The number of negative eigenvalues  $n(\mathcal L^\gamma_1)$ for $\gamma >0$} 

\noindent Below  we will show the following result.
\begin{proposition}\label{num} Let  $\gamma>0$, then $n(\mathcal L^\gamma_1)= 1$. Moreover, for $\mathcal L^\gamma_1$ restricted to  $\widetilde{W}_{\rad}$  we also have $n(\mathcal L^\gamma_1)= 1$.  In particular, the unique  negative simple eigenvalue equals $-2$, and  $\varphi_{\omega,\gamma}$ is the corresponding eigenfunction.  
\end{proposition}
\noindent
The proof of this proposition relies on  the theory of extension of symmetric operators.
We start with two preliminary  results.
\begin{lemma} \label{ALK}
The  operator defined by 
\begin{equation}\label{L_0}
\begin{split}
 &\mathcal L_0 =-\frac{d^2}{dx^2}+ \Big(|x|+\frac{\gamma}{2}\Big)^2-3,\\
& \dom(\mathcal L_0) =\big\{ f\in  H^2(\mathbb R): x^2f\in L^2(\mathbb R),  f(0)=0 \big\}
\end{split}
\end{equation}
  is a densely defined symmetric operator with equal  deficiency indices  \\ $n_{\pm}(\mathcal L_0)=1$.
\end{lemma} 
\begin{proof} First, we  establish the  scale of Hilbert spaces associated with the self-adjoint operator  (see \cite[Section I,\S 1.2.2]{ak})
$$\mathcal L=-\frac{d^2}{dx^2}+ \Big(|x|+\frac{\gamma}{2}\Big)^2,\quad\dom(\mathcal L)=\{ f\in  H^2(\mathbb R): x^2f\in L^2(\mathbb R)\}.$$
Define for $s\geq 0$ the space 
$$
\mathscr H_s(\mathcal L)=\big\{f\in L^2(\mathbb R): \|f\|_{s,2}=\Big\|(\mathcal L+\mathcal I)^{s/2}f\Big\|_2<\infty\big\}.
$$
The space $\mathscr H_s(\mathcal L)$ with norm $\|\cdot\|_{s,2}$ is complete. The dual space  of $\mathscr H_s(\mathcal L)$ will be denoted by $\mathscr H_{-s}(\mathcal L)=\mathscr H_s(\mathcal L) '$. The norm in the space $\mathscr H_{-s}(\mathcal L)$ is defined by the formula
$$
\|\psi\|_{-s,2}=\Big \|(\mathcal L+\mathcal I)^{-s/2} \psi\Big\|_2.
$$
The spaces $\mathscr H_s(\mathcal L)$ form the following chain 
\begin{equation*}\label{triples}
...\subset \mathscr H_2(\mathcal L)\subset \mathscr H_1(\mathcal L)\subset L^2(\mathbb R)=\mathscr H_0(\mathcal L)\subset\mathscr H_{-1}(\mathcal L)\subset \mathscr H_{-2}(\mathcal L)\subset....
\end{equation*}
Thus, the  space $\mathscr H_2(\mathcal L)$ coincides with the domain of the operator $\mathcal L$. The norm of the space 
$\mathscr H_1(\mathcal L)$ can be calculated as follows
\begin{equation*}
\begin{split}
&\|f\|^2_{1,2}=((\mathcal L+\mathcal I)^{1/2}f, (\mathcal L+\mathcal I)^{1/2}f)\\&=\int\limits_{\mathbb{R}}\left( |f'(x)|^2 +|f(x)|^2 +\Big(|x|+\frac{\gamma}{2}\Big)^2|f(x)|^2\right)dx.
\end{split}
\end{equation*}
Therefore, we have the embedding $\mathscr H_1(\mathcal L) \hookrightarrow  H_1(\mathbb R)$ and, by Sobolev embedding, $\mathscr H_1(\mathcal L) \hookrightarrow L^\infty (\mathbb R)$. From the former remark we obtain that the $\delta$-functional, $\delta: \mathscr H_1(\mathcal L)\to \mathbb C$ acting as  $\delta (\psi)=\psi(0)$ belongs to $\mathscr H_1(\mathcal L)'=\mathscr H_{-1}(\mathcal L)$ and consequently  $\delta\in \mathscr H_{-2}(\mathcal L)$. Therefore, using   \cite[Lemma 1.2.3]{ak}, it follows  that the restriction $\mathcal L_0$ of the operator $\mathcal L$  to the domain 
$$
\dom(\mathcal L'_0)=\{\psi\in \dom(\mathcal L): \delta (\psi)=\psi(0)=0\}
$$ 
is a densely defined symmetric operator with  equal deficiency indices  \\ $n_{\pm}(\mathcal L'_0)=1$. Next, since $\mathcal B=-3\mathcal I$ is a bounded operator, we have from   \cite[Chapter IV, Theorem 6]{Nai67} that the operators $\mathcal L'_0$ and $\mathcal L_0= \mathcal L'_0+\mathcal B$ have the same deficiency indices. This finishes the proof of the Lemma.
\end{proof}
 To investigate the number of negative eigenvalues of $\mathcal L_1^\gamma$ we will use the following abstract result (see \cite[Chapter IV, \S 14]{Nai67}).
\begin{proposition}\label{semibounded}
Let $\mathcal A$  be a densely defined lower semi-bounded symmetric operator (i.e., $\mathcal A\geq m\mathcal I$)  with finite deficiency indices $n_{\pm}(\mathcal A)=k<\infty$  in the Hilbert space $\mathfrak{H}$. Let also $\widetilde{\mathcal A}$ be a self-adjoint extension of $\mathcal A$.  Then the spectrum of $\widetilde{\mathcal A}$  in $(-\infty, m)$ is discrete and  consists of at most  $k$  eigenvalues counting multiplicities.
\end{proposition}
\begin{remark}
Proposition \ref{semibounded} holds for upper semi-bounded operator $\mathcal A$  ($\mathcal A\leq M\mathcal I$) and interval $(M, \infty)$, respectively.
\end{remark}

\noindent\textit{Proof of Proposition \ref{num}.}
Recall that  $\mathcal L^\gamma_1$ is the self-adjoint extension of the  symmetric  operator $\mathcal{L}_0$ defined by \eqref{L_0} (see proof of Theorem \ref{repres} above).
Lemma \ref{ALK} implies the equality $n_{\pm}(\mathcal L_0)=1$. 
 
 Next,  since $\gamma>0$ ($\varphi'_{\omega,\gamma} \neq 0$ for $x\neq 0$) we can verify that for $f\in \dom(\mathcal L_0)$ we have (see \cite[Subsection 6.1]{AdaNoj13a})
 \begin{equation}\label{NLSL11}
-f''+ \left[\Big(|x|+\frac{\gamma}{2}\Big)^2-3\right]f=\frac{-1}{\varphi'_{\omega,\gamma}}\frac{d}{dx}\left[(\varphi'_{\omega,\gamma})^2\frac{d}{dx}\left(\frac{f}{\varphi'_{\omega,\gamma}}\right)\right],\quad\quad x\neq 0.
\end{equation}
 Now using \eqref{NLSL11} and integrating by parts, we get
\begin{equation}\label{NLSL12}
\begin{split}
(\mathcal L_{0} f,f)=&
\int\limits_{-\infty}^{0-}(\varphi'_{\omega,\gamma})^2\left|\frac{d}{dx}\left(\frac{f}{\varphi'_{\omega,\gamma}}\right)\right|^2dx\\ &+
\int\limits^{\infty}_{0+}(\varphi'_{\omega,\gamma})^2\left|\frac{d}{dx}\left(\frac{f}{\varphi'_{\omega,\gamma}}\right)\right|^2dx+
\left[f'\overline{f}-|f|^2\frac{\varphi''_{\omega,\gamma}}{\varphi'_{\omega,\gamma}}\right]_{0-}^{0+}.
\end{split}
\end{equation}
The integral terms in \eqref{NLSL12} are nonnegative.  Due to the condition $f(0)=0$, non-integral term vanishes, and we get $\mathcal L_{0}\geq 0$. Therefore, from Proposition \ref{semibounded} we obtain $n(\mathcal L^\gamma_1)\leq 1$. 
From the other hand,    
 \begin{equation}\label{ineq1}
 \mathcal L^\gamma_1 \varphi_{\omega, \gamma}=(\mathcal L^\gamma_2-2)\varphi_{\omega, \gamma}=-2\varphi_{\omega, \gamma},
  \end{equation}
since  $\mathcal L^\gamma_2 \varphi_{\omega, \gamma}=0$. Thus,  $n(\mathcal L^\gamma_1)= 1$.
 The second assertion of Proposition \ref{num} follows from \eqref{ineq1} and the fact that  $\varphi_{\omega, \gamma}$ is even.
\hfill$\square$
\vspace{0.5cm}

\noindent{\bf The number of negative eigenvalues $n(\mathcal L^\gamma_1)$ for $\gamma <0$}

The  analysis previously applied to calculate the number $n(\mathcal L^\gamma_1)$ was based essentially on the fact that  $\varphi'_{\omega, \gamma}(x) \neq 0$ for $x\neq 0$ in the case  $\gamma>0$. For $\gamma<0$ the function  $\varphi'_{\omega, \gamma}(x)$ has exactly two zeroes $x=\pm \frac{\gamma}{2}$, and the  formula \eqref{NLSL11} could not be applied. To study the case  of negative $\gamma$  we will use the theory of analytic perturbations for linear operators (see \cite{kato, RS}).

The following lemma states the analyticity of the families of operators $\mathcal L^\gamma_{j},\,j\in\{1,2\}$. 
\begin{lemma}\label{analici} As a function of $\gamma$, $(\mathcal L_{1}^\gamma)$ and $(\mathcal L_{2}^\gamma)$ are two real-analytic families of self-adjoint operators of type (B) in the sense of Kato.
\end{lemma}
\begin{proof} By Theorem \ref{repres} and \cite[Theorem VII-4.2]{kato}, it suffices to prove that the families of bilinear forms $(B^{\gamma}_1)$ and $(B^{\gamma}_2)$ defined in \eqref{spec14} are  real-analytic  of type (B). Indeed,  it is immediate that they are bounded  from below and closed. Moreover, the decomposition  of $B^{\gamma}_1$ into $B^{\gamma}$ and $B_1$,  implies that $\gamma\to (B^{\gamma}_1v, v)$ is  analytic.  The proof for the family $(B^{\gamma}_2)$ is similar.  
\end{proof}
In what follows we also use the following classical result about the harmonic oscillator operator \eqref{osci} (see \cite{BerShu91}).

\begin{lemma}\label{lemma_L0} Let operator  $\mathcal L^0_1$ be defined by \eqref{osci}. Then the following assertions hold. 
\begin{itemize}
\item[$(i)$] $\mathcal L^0_1$ has two simple nonpositive eigenvalues: the first one is negative and the second one is zero.
\item[$(ii)$]  $\ker(\mathcal L^0_1)=\Span\{\varphi_{\omega,0}'\}$.
\item[$(iii)$] The rest of the spectrum of $\mathcal L^0_1$ is positive.
\end{itemize}
\end{lemma}
Indeed, the above Lemma follows from the known fact $\sigma(\mathcal L^0_1)=\{2n-2: n=0,1,2,...\}$.

\begin{proposition}\label{perteigen}  There  exist $\gamma_0>0$ and two analytic functions \\ $\Pi : (-\gamma_0,\gamma_0)\to \mathbb R$ and $\Omega: (-\gamma_0,\gamma_0)\to L^2(\mathbb{R})$ such that
\begin{enumerate}
\item[$(i)$] $\Pi(0)=0$ and $\Omega(0)=\varphi_{\omega,0}'$.

\item[$(ii)$] For all $\gamma\in (-\gamma_0,\gamma_0)$, $\Pi(\gamma)$ is the  simple isolated second eigenvalue of $\mathcal L_{1}^\gamma$, and $\Omega(\gamma)$ is the associated eigenvector for $\Pi(\gamma)$.

\item[$(iii)$] $\gamma_0$ can be chosen small enough to ensure that for  $\gamma\in (-\gamma_0,\gamma_0)$  the spectrum of $\mathcal L_{1}^\gamma$ is positive, except at most the  first two eigenvalues.
\end{enumerate}
\end{proposition}
\begin{proof} Using the spectral structure of the operator $\mathcal L^0_1$ (see Lemma \ref{lemma_L0}), we can separate the spectrum $\sigma(\mathcal L^0_1)$   into two parts $\sigma_0=\{\lambda^0_1, 0\}$ and  $\sigma_1$ by a closed curve  $\Gamma$ (for example, a circle), such that $\sigma_0$ belongs to the inner domain of $\Gamma$ and $\sigma_1$ to the outer domain of $\Gamma$ (note that $\sigma_1\subset (\epsilon, +\infty)$ for $\epsilon>0$).  Next,  Lemma \ref{analici} and analytic perturbations theory imply  that  $\Gamma\subset \rho(\mathcal L^\gamma_{1})$ for sufficiently small $|\gamma |$, and $\sigma (\mathcal L^\gamma_{1})$ is likewise separated by $\Gamma$ into two parts, such   that the part of $\sigma (\mathcal L^\gamma_{1})$ inside $\Gamma$ consists of a finite number of eigenvalues with total multiplicity (algebraic) two. Therefore, we obtain from the Kato-Rellich Theorem (see  \cite[Theorem XII.8]{RS}) the existence of two analytic functions $\Pi, \Omega$ defined in a neighborhood of zero such that  the items $(i)$, $(ii)$ and $(iii)$ hold.
\end{proof}

Below we will study how the perturbed second eigenvalue $\Pi(\gamma)$ changes depending on the sign of $\gamma$.  For small $\gamma$  we have the following result.
\begin{proposition}\label{signeigen} There exists $0<\gamma_1<\gamma_0$ such that $\Pi(\gamma)<0$ for any $\gamma\in (-\gamma_1,0)$, and $\Pi(\gamma)>0$ for  any $\gamma\in (0, \gamma_1)$. 
\end{proposition}

\begin{proof} From Taylor's theorem we have the following expansions
\begin{equation}\label{decomp1}
\Pi(\gamma)=\beta \gamma+ O(\gamma ^2)\quad\text{and}\quad \Omega(\gamma)=\varphi'_{\omega,0}+ \gamma \psi_0 +  O(\gamma^2),
\end{equation}
where  $\beta\in \mathbb R$ ($\beta=\Pi'(0)$) and $\psi_0\in L^2(\mathbb{R})$ (since $\psi_0=\Omega'(0)$).  The desired result will follow  if we show that $\beta>0$.  We compute $(\mathcal L^\gamma_{1} \Omega(\gamma), \varphi'_{\omega,0})$ in two different ways. 

 From \eqref{decomp1} we obtain
\begin{equation}\label{produ}
\Pi(\gamma) \Omega(\gamma)=\beta \gamma \varphi'_{\omega,0} + O(\gamma^2).
\end{equation}
Since $\mathcal L^\gamma_{1} \Omega(\gamma)=\Pi(\gamma) \Omega(\gamma)$, it follows from \eqref{produ} that
\begin{equation}\label{decomp5}
\begin{aligned}
(\mathcal L^\gamma_{1} \Omega(\gamma),  \varphi'_{\omega,0})=\beta \gamma\|\varphi'_{\omega,0}\|_2^2 +O(\gamma^2).
\end{aligned}
\end{equation}
Having $\varphi'_{\omega,0}\in \mathcal \dom(\mathcal L^\gamma_{1})$ ($\varphi'_{\omega,0}(0)=0$) and $\mathcal L_1^0\varphi'_{\omega,0}=0$, we obtain
\begin{equation}\label{produ2}
 \mathcal L^\gamma_{1}\varphi'_{\omega,0}=\mathcal L_1^0\varphi'_{\omega,0} + (\gamma |x| + \tfrac{\gamma^2}{4})\varphi'_{\omega,0} =(\gamma |x| + \tfrac{\gamma^2}{4})\varphi'_{\omega,0}.
\end{equation}
Since $\mathcal L^\gamma_{1} $ is  self-adjoint,  we obtain   from \eqref{decomp1} and \eqref{produ2}  that 
\begin{equation}\label{pro3}
\begin{aligned}
(\mathcal L^\gamma_{1} \Omega(\gamma), \varphi'_{\omega,0})&=(\Omega(\gamma), \mathcal L^\gamma_{1} \varphi'_{\omega,0})=(\varphi'_{\omega,0},  (\gamma |x| + \tfrac{\gamma^2}{4})\varphi'_{\omega,0}) + O(\gamma^2)\\
&= \gamma \int\limits_\mathbb{R} |x||\varphi'_{\omega,0}(x)|^2 dx + O(\gamma^2).
\end{aligned}
\end{equation}
Finally,  combination of \eqref{decomp5} and \eqref{pro3} leads to 
\begin{equation}\label{decomp10}
\beta=\frac{\int\limits_\mathbb{R} |x||\varphi'_{\omega}(x)|^2 dx}{\|\varphi'_{\omega,0}\|_2^2}+O(\gamma).
\end{equation}
From \eqref{decomp10} it follows that $\beta>0$, and, therefore, assertion is proved.
\end{proof}

Now we can count  the number of negative eigenvalues of $\mathcal L^\gamma_{1}$ for any  $\gamma$  using  a classical continuation argument based on the Riesz-projection. 
\begin{proposition}\label{spec6a} Let $\gamma \in \mathbb R\setminus \{0\}$.  Then we have
\begin{itemize}
\item[$(i)$] $n(\mathcal L^\gamma_{1})=2$  for $\gamma<0$.
\item[$(ii)$]  $n(\mathcal L^\gamma_{1})=1$ for $\gamma>0$.
\item[$(iii)$]  $n(\mathcal L^\gamma_{1})=1$ for $\mathcal L^\gamma_{1}$ restricted to $\widetilde{W}_{\rad}$.
\end{itemize}
\end{proposition}
\begin{proof} Recall that  $\ker(\mathcal L^\gamma_{1})=\{0\}$ for $\gamma\neq 0$. Let $\gamma<0$ and define $\gamma_\infty$ by
$$
\gamma_\infty=\inf \{r<0:  \mathcal L^\gamma_{1}\;{\text{has exactly two negative eigenvalues for all}}\; \gamma \in (r,0)\}.
$$
 Proposition \ref{signeigen} implies that $\gamma_\infty$ is well defined and $\gamma_\infty\in [-\infty,0)$. We claim that $\gamma_\infty=-\infty$. Suppose that $\gamma_\infty> -\infty$. Let $N=n(\mathcal L^{\gamma_{\infty}}_{1})$ and $\Gamma$  be a closed curve (for example, a circle or a rectangle) such that $0\in \Gamma\subset \rho(\mathcal L^{\gamma_{\infty}}_{1})$, and  all the negative eigenvalues of  $\mathcal L^{\gamma_{\infty}}_{1}$ belong to the inner domain of $\Gamma$.  The existence of such $\Gamma$ can be  deduced from the lower semi-boundedness of the quadratic form associated to $\mathcal L^{\gamma_{\infty}}_{1}$. Indeed, for $f\in \dom(\mathcal L^{\gamma_{\infty}}_{1})$
$$
(\mathcal L^{\gamma_{\infty}}_{1} f, f)=\int\limits_{\mathbb{R}}((f')^2 +V^1_{\gamma_{\infty}} f^2)dx -\gamma|f(0)|^2\geq -3\|f\|_2^2
$$
since $V^1_{\gamma_{\infty}}(x)\geq -3$ for all $x$.

Next, from Lemma \ref{analici} it  follows  that there is  $\epsilon>0$ such that for $\gamma\in [\gamma_{\infty}-\epsilon, \gamma_{\infty}+\epsilon]$ we have $\Gamma\subset \rho(\mathcal L^\gamma_{1})$ and for $\xi \in \Gamma$,
$\gamma\to ( \mathcal L^\gamma_{1}-\xi)^{-1}$ is analytic. Therefore, the existence of an analytic family of Riesz-projections $\gamma\to P(\gamma)$  given by 
$$
P(\gamma)=-\frac{1}{2\pi i}\int\limits_{\Gamma} (\mathcal L^{\gamma}_{1}-\xi)^{-1}d\xi
$$
implies  that $\dim(\ran P(\gamma))=\dim(\ran P(\gamma_\infty))=N$ for all $\gamma\in [\gamma_{\infty}-\epsilon, \gamma_{\infty}+\epsilon]$. Next, by definition of $\gamma_{\infty}$, there exists $r_0\in (\gamma_\infty, \gamma_\infty+\epsilon)$, and $ \mathcal L^\gamma_{1}$ has exactly two negative eigenvalues for all $\gamma\in (r_0,0)$. Therefore, $\mathcal L^{\gamma_\infty+\epsilon}_{1} $ has two negative eigenvalues and $N=2$, hence $\mathcal L^\gamma_{1}$ has two negative eigenvalues for $\gamma\in (\gamma_{\infty}-\epsilon, 0)$,  which contradicts with the definition of $\gamma_{\infty}$. Therefore,  $\gamma_{\infty}=-\infty$. Similar analysis can be  applied to the case $\gamma>0$. The last assertion was proved for $\gamma>0$ in  Proposition \ref{num}. In the case $\gamma<0$ the statement follows from  item $(i)$, the fact that any eigenfunction of $\mathcal L^{\gamma}_1$ is either even or odd, and  the Sturm oscillation result in Lemma \ref{Oscill}. 
\end{proof}

\begin{remark} We note that the curve $\Gamma$ above  can be chosen independently of the parameter $\gamma \in \mathbb R$. Indeed, the relation $V_1^{\gamma}(x)\geq -3$ for any $\gamma$ implies $\inf\sigma(\mathcal L^\gamma_1)\geq -3$.  Thus, $\Gamma$ can be chosen as the rectangle $\Gamma=\partial R$, in which $$
R=\{ z\in\mathbb C: z=z_1+iz_2, (z_1,z_2)\in [-4,0]\times [-a,a],\;\text{for some}\;a>0\}.$$
\end{remark}
\vskip0.2in
\noindent\textit{Proof of Theorem \ref{main}}
\begin{itemize} 
\item[$(i)$] Let $\gamma>0$ and  $E: \widetilde W\longrightarrow\mathbb{R}$ be the energy functional defined by \eqref{energy}. From  \cite[Lemma 2.6]{Caz83} (with $-\Delta$ substituted by $\mathcal H^\gamma_\delta$) and the  continuous embedding $\widetilde{W}\hookrightarrow W(\mathbb{R})$ we deduce that   $\mathcal E''(\varphi_{\omega,\gamma})\in B(\widetilde{W}, \widetilde{W}')$, where $\mathcal E''(\varphi_{\omega,\gamma})$  is the operator associated with the  form 
$E''(\varphi_{\omega,\gamma})(u,v)$. Using  Proposition \ref{num}, positivity of $\partial_\omega||\varphi_{\omega, \gamma}||^2_{2}$, Remark \ref{morse}, and Theorem \ref{ass3} we arrive   at  item $(i)$  in Theorem \ref{st0} which induces  the orbital stability of $e^{i\omega t}\varphi_{\omega,\gamma}$ in $\widetilde{W}$.

\item[$(ii)$] Let $\gamma<0$. From item $(i)$ of Proposition \ref{spec6a}, the positivity  of $\partial_\omega||\varphi_{\omega, \gamma}||^2_{2}$, and Theorem \ref{ass3},  we get   item $(ii)$ of Theorem  \ref{st0} which implies the instability of $e^{i\omega t}\varphi_{\omega,\gamma}$ in $\widetilde{W}$.
\item[$(iii)$] Stability of  $e^{i\omega t}\varphi_{\omega,\gamma}$ in $\widetilde{W}_{\rad}$ follows from item $(iii)$ of Propositon \ref{spec6a} and item $(i)$  in Theorem \ref{st0}.
\end{itemize}
\hfill
$\square$

\section*{Appendix}

In this Appendix we show the uniqueness of the peak-standing wave solution $\varphi_{\omega, \gamma}$ stated in Theorem \ref{uniq}. The  proof is based on the ideas from \cite{Ard16, BM, FJ, Vaz84}.

\noindent\textit{Proof of Theorem \ref{uniq}.} We divide the proof in 3 steps. Let $\varphi$ be a solution to \eqref{eq2}.
\begin{enumerate}
\item[1)] We show initially that if $\varphi\in H^2(\mathbb{R}_+)$ is a solution to
\begin{equation}\label{prof1}
-\varphi''+\omega\varphi-\varphi\log|\varphi|^2=0
\end{equation}
on $\mathbb{R}_+$, then  $\varphi=e^{i\theta_+}e^{\tfrac{\omega+1}{2}}e^{-\tfrac{(x-x_+)^2}{2}}$, where $\theta_+, x_+\in\mathbb{R}$. Indeed,  writing  $\varphi(x)= e^{i\theta(x)}\rho(x)$, where $\theta$ and $\rho$ are  real-valued functions, we obtain from equation \eqref{prof1}
  $$ 
  -\rho''+\rho(\omega+(\theta')^2)-\rho\log\rho^2+i(\theta''\rho+2\theta'\rho')=0.
 $$
 Thus, in order to make the
imaginary part vanish, we get $\theta''\rho+2\theta'\rho'=0$, which implies $\rho^2\theta'\equiv const:=C$. Next, since 
$$
|\varphi'|^2=(\rho')^2+(\theta')^2\rho^2\geq (\theta')^2\rho^2\geq 0
$$
 and $\underset{x\rightarrow\infty}\lim|\varphi'|=0$, we get 
$\underset{x\rightarrow\infty}\lim (\theta')^2\rho^2=C\underset{x\rightarrow\infty}\lim \theta'=0$. 
Therefore, $\underset{x\rightarrow\infty}\lim \theta'$ exists. Now, since $|\varphi|=\rho$, we obtain  $\underset{x\rightarrow\infty}\lim \rho^2=0$, and thus  $C=0$,  which implies  $\theta(x)\equiv const:=\theta_+$.
Thereby, $\varphi(x)= e^{i\theta_+}\rho(x)$, where $\rho$ satisfies 
\begin{equation}\label{prof2}
 -\rho''+\omega\rho-\rho\log\rho^2=0,\quad x\in\mathbb{R}_+.
\end{equation}
From \cite[Theorem 1]{Vaz84} it follows that $\rho(x)>0$.

\item[2)] Multiplying \eqref{prof2} by $\rho'$ and integrating we arrive at 
\begin{equation}\label{prof3}
(\rho')^2=(1+\omega)\rho^2-\rho^2\ln|\rho|^2+K.
\end{equation}
Since $\rho\in H^2(\mathbb{R}_+)$, we get $K=0$. Therefore, integrating \eqref{prof3} we obtain 
$$
\rho(x)=e^{\tfrac{\omega+1}{2}}e^{-\tfrac{(x-x_+)^2}{2}},\qquad x_+\in\mathbb{R}.
$$ 
Thus, we get $\varphi=e^{i\theta_+}e^{\tfrac{\omega+1}{2}}e^{-\tfrac{(x-x_+)^2}{2}}$ on $\mathbb{R}_+$.
Analogously, we can show that the $H^2(\mathbb{R}_-)$-solution of \eqref{prof1} on $\mathbb{R}_-$ is given by 
$$\varphi=e^{i\theta_-}e^{\tfrac{\omega+1}{2}}e^{-\tfrac{(x-x_-)^2}{2}},\qquad\theta_-, x_-\in\mathbb{R}.
$$ 

\item[3)]  From items 1)-2) above we obtain that the solution to \eqref{prof1} on $\mathbb R\setminus\{0\}$ is given by 
$$
\varphi=\begin{cases}
\begin{aligned}
& e^{i\theta_+}e^{\tfrac{\omega+1}{2}}e^{-\tfrac{(x-x_+)^2}{2}},\,\,x>0,\\
& e^{i\theta_-}e^{\tfrac{\omega+1}{2}}e^{-\tfrac{(x-x_-)^2}{2}},\,\, x<0.
\end{aligned}
\end{cases}
$$
Next,  our aim is to find explicitly $x_\pm$ and $\theta_\pm$. Let $f(s) = -\omega s+s \log(s^2)$ and   $F(s) =\int\limits_0^s f(t)dt$.
Multiplying \eqref{prof1} by $\varphi'$ and integrating from $0$  to $R$, we get as $R\rightarrow\infty$
$$\frac{1}{2}(\varphi'(0+))^2 + F(\varphi(0+)) = 0. $$
Similarly, we obtain 
$$\frac{1}{2}(\varphi'(0-))^2 + F(\varphi(0-)) = 0. $$
Since $\varphi$  need to be continuous at $x=0$, we get $|\varphi'(0-)|=|\varphi'(0+)|$. Therefore, $|x_-|=|x_+|$ and again by continuity  condition we obtain $\theta_-=\theta_+=:\theta$. To conclude the proof we need to recall that $\varphi$ satisfies jump condition  $\varphi'(0+)-\varphi'(0-)=-\gamma\varphi(0)$, which yields  $x_+=-\frac{\gamma}{2}$ and $x_-=\frac{\gamma}{2}$. 
\end{enumerate}
\hfill
$\square$

\end{document}